\documentclass[a4paper,11pt]{amsart}
\usepackage{amsfonts,stmaryrd,amssymb,amsmath,a4wide,verbatim}
\usepackage[british]{babel}
\usepackage[active]{srcltx}
\ifx\pdfoutput\undefined
    \usepackage[pdftex]{graphicx}
    \DeclareGraphicsExtensions{.jpg, .png , .pdf, .eps}
\else
    \usepackage[dvips]{graphicx}
    \DeclareGraphicsExtensions{.eps, .ps, .eps.gz, ps.gz}
\fi

\makeatletter
\@addtoreset{equation}{section}\makeatother

\newcommand{\normdpoly}[1]{\left\| #1\right\|_{\S^n}}
\newcommand{\normdfonc}[1]{\| #1\|_2}

\def\tr{{\rm tr}\,}

\newtheorem{theorem}{Theorem}[section]

\newtheorem{lemma}[theorem]{Lemma}
\newtheorem{corollary}[theorem]{Corollary}

\def\R{\mathbb{R}}
\def\N{\mathbb{N}}
\def\S{\mathbb{S}}
\def\insm{\displaystyle\int_{M}}

\def\vol{dv}

\def\Vol{{\rm Vol}\,}
\def\H{{\rm H}}
\def\B{{\rm B}}
\def\xb{\overline{X}}

\def\la{\lambda_1}

\def\hkm{{\mathcal H}^k(M)}
\def\hkr{{\mathcal H}^k(\R^{n+1})}

\begin{document}
\title[]{Metric shape of hypersurfaces with small extrinsic radius or large $\lambda_1$}



\author[E. AUBRY, J.-F. GROSJEAN]{Erwann AUBRY, Jean-Fran\c cois GROSJEAN}

\address[E. Aubry]{LJAD, Universit\'e de Nice Sophia-Antipolis, CNRS; 28 avenue Valrose, 06108 Nice, France}
\email{eaubry@unice.fr}

\address[J.-F. Grosjean]{Institut \'Elie Cartan de Lorraine (Math\'ematiques), Universit\'e de Lorraine, B.P. 239, F-54506 Vand\oe uvre-les-Nancy cedex, France}
\email{jean-francois.grosjean@univ-lorraine.fr}

\date{\today}

\begin{abstract} We determine the Hausdorff limit-set of the Euclidean hypersurfaces with large $\lambda_1$ or small extrinsic radius. The result depends on the $L^p$ norm of the curvature that is assumed to be bounded a priori, with a critical behaviour for $p$ equal to the dimension minus $1$.
\end{abstract}

\maketitle

\section{Introduction}

For any $A\subset\R^{n+1}$ and any $\varepsilon>0$, we set $A_\varepsilon$ the tubular neighbourhood of radius $\varepsilon$ of $A$ ($A_\varepsilon=\{x\in\R^{n+1}/\,d(A,x)\leqslant\varepsilon\}$). 
$d_H(A,B)=\inf\{\varepsilon>0\,A\subset B_\varepsilon$ and $B\subset A_\varepsilon\}$ is called the Hausdorff distance on  closed subsets of $\R^{n+1}$. 
Let $(M_k^m)_{k\in\N}$ be a sequence of immersed submanifolds of dimension $m$ in $\R^{n+1}$. 
We say that it converges weakly to a subset $Z\subset\R^{n+1}$ in Hausdorff topology if there exists a sequence of subsets $A_k\subset M_k$ such that $d_H(A_k,Z)\to 0$ and $\Vol(M_k\setminus A_k)/\Vol M_k\to 0$.

Of course, weak Hausdorff convergence does not imply Hausdorff convergence without supplementary assumption. Our first aim will be to determine which $L^p$-norm of the mean curvature has to be bounded so that weak convergence implies convergence. More precisely, we will study the limit-set for the Hausdorff distance of a weakly converging sequence  of submanifolds with $L^p$ norm of the mean curvature uniformly bounded and show that it depends essentially on the value of $p$.
As an application, we derive some new results on the metric shape of Euclidean hypersurfaces with small extrinsic radius or large $\lambda_1$.

\subsection{Weak Hausdorff convergence vs Hausdorff convergence}

In the paper, the $L^p$-norms are defined by $\|f\|_p^p=\frac{1}{v_M}\int_M|f|^p\vol$.  We denote by $m_1$ the 1-dimensional Hausdorff measure on $\R^{n+1}$. We denote by $\B$ the second fundamental form and $\H=\frac{1}{m}\tr \B$ the mean curvature of Euclidean $m$-submanifolds.

Our main result says that if $\Vol M_k\|\H\|_p^{m-1}$ remain bounded for some $p>m-1$ (resp. for $p=m-1$) then weak Hausdorff convergence implies Hausdorff convergence (resp. up to a set of bounded $1$-dimensional Hausdorff measure).

\begin{theorem}\label{Main}
Let $(M_k)_{k\in\N}$ be a sequence of immersed, compact submanifolds of dimension $m$ which weakly converges to $Z\subset\R^{n+1}$.

If there exist $p>m-1$ and $A>0$ such that $\Vol (M_k)\|\H\|^{m-1}_p\leqslant A$ for any $k$, then $d_H(M_k,Z)\to 0$.

There exists a constant $C(m)$ such that if $\Vol (M_k)\|\H\|^{m-1}_{m-1}=\int_{M_k}|\H|^{m-1}\leqslant A$ for any $k$, then the limit-set of $(M_k)_{k\in\N}$ for the Hausdorff distance is not empty and any limit point is a closed, connected subset $Z\cup T\subset\R^{n+1}$ such that $m_1(T)\leqslant C(m)A$.
\end{theorem}

Note that it derives from the proof that in the case $p=m-1$, we have $m_1(T)\leqslant C(m)\max_{\varepsilon}\liminf_{k}\int_{M_k\setminus Z_\varepsilon}|\H|^{m-1}$.

The previous result is rather optimal as shows the following result.

\begin{theorem}\label{constrfond}
 Let $M_1,M_2\hookrightarrow\R^{n+1}$ be two immersed compact submanifolds with the same dimension $m$, $M_1\#M_2$ be their connected sum and $T$ be any closed subset of $\R^{n+1}$ such that $M_1\cup T$ is connected. Then there exists a sequence of immersions $i_k:M_1\#M_2\hookrightarrow\R^{n+1}$ such that\\

1) $i_k(M_1\# M_2)$ weakly converges to $M_1$ and converges to $M_1\cup T$ in Hausdorff topology,

2) the curvatures of $i_k(M_1\# M_2)$ satisfy
\begin{align*}
&\int_{i_k(M_1\#M_2)}|\H|^{m-1}\to \int_{M_1}|\H|^{m-1}+(\frac{m-1}{m})^{m-1}\Vol\S^{m-1}m_1(T),\\
&\int_{i_k(M_1\#M_2)}|\B|^{m-1}\to \int_{M_1}|\B|^{m-1}+\Vol\S^{m-1}m_1(T),\\
&\int_{i_k(M_1\#M_2)}|\H|^\alpha\to \int_{M_1}|\H|^\alpha\quad\quad\mbox{for any }\alpha\in[1,m-1),\\
&\int_{i_k(M_1\#M_2)}|\B|^\alpha\to \int_{M_1}|\B|^\alpha\quad\quad\mbox{for any }\alpha\in[1,m-1),
\end{align*}

3) $\lambda_p(i_k(M_1\#M_2))\to\lambda_p(M_1)$ for any $p\in\N$,

4) $\Vol(i_k(M_1\#M_2))\to\Vol M_1$.
\end{theorem}

Conditions 3) and 4) imposed to our sequence of immersions in Theorem \ref{constrfond} are designed on purpose for our study of almost extremal Euclidean hypersurfaces for the Reilly or Hasanis-Koutroufiotis Inequalities.

Theorem \ref{Main} proves that for $p>m-1$ the Hausdorff limit-set of a weakly convergent sequence is reduced to the weak limit. On the contrary, Theorem \ref{constrfond} shows that for $p<m-1$, the Hausdorff limit-point of a weakly convergent sequence can be any closed, connected Euclidean subset containing the weak-limit. For the critical exponent $p=m-1$, the Hausdorff limit-set can contain any $Z\cup T$ with $m_1(T)\leqslant C_2(m)A$ (by Theorem \ref{constrfond}) and contains only $Z\cup T$ with $m_1(T)\leqslant C_1(m)A$ (by Theorem \ref{Main}). Unfortunately, our constants $C_1(m)$ and $C_2(m)$ are different. We conjecture that is is only due to lake of optimality of the constant in Theorem \ref{Main}.

Note that the two previous theorems can be easily extended to the case where $\R^{n+1}$ is replaced by any fixed Riemannian manifold $(N,g)$.
 
\subsection{Application to hypersurfaces with large $\lambda_1$ or small Extrinsic radius}

Let $X{:}\,M^n\to\R^{n+1}$ be a closed, connected, immersed Euclidean hypersurface (with $n\geqslant 2)$. We set $v_M$ its volume and $\xb:=\frac{1}{v_M}\int_MX\vol$ its center of mass.

The Hasanis-Koutroufiotis inequality is the following lower bound on the extrinsic radius $r_{M}$ of $M$ (i.e. the least radius of the Euclidean balls containing $M$)
\begin{equation}\label{rext}
 r_{M}\|\H\|_2\geqslant1.
\end{equation}
This inequality is optimal since we have equality for any Euclidean sphere. Moreover, if an immersed hypersurface $M$ satisfies the equality case then $M$ is the Euclidean sphere $S_M=\overline{X}+\frac{1}{\|\H\|_2}\S^n$  with center $\overline{X}$ and radius $\frac{1}{\|\H\|_2}$.

The Reilly inequality is the following upper bound on the first non zero eigenvalue $\lambda^M_1$ of $M$
\begin{equation}\label{lambda}
\la^M\leqslant n\|\H\|_2^2,
\end{equation}
once again we have equality if and only if $M$ is the sphere $S_M$. 

Our aim is to study the metric shape of the Euclidean hypersurfaces with almost extremal extrinsic radius or $\lambda_1$. 

\subsubsection{Almost extremal hypersurfaces weakly converge to $S_M$}

Our first result describes some volume and curvature concentration properties of almost extremal hypersurfaces that imply weak convergence to $S_M$. Note that in this result we do not assume any bound on the mean curvature. 

We set $B_x(r)$ the closed ball with center $x$ and radius $r$ in $\R^{n+1}$ and $A_\eta$ the annulus $\bigl\{X\in\R^{n+1}/\bigl|\|X-\bar{X}\|-\frac{1}{\|\H\|_2}\bigr|\leqslant\frac{\eta}{\|\H\|_2}\bigr\}$.
Throughout the paper we shall adopt the notation that $\tau(\varepsilon|n,p,h,\cdots)$ is a positive function which depends on $n,p,h,\cdots$
and which converges to zero as $\varepsilon\to 0$. These functions $\tau$ will always
be explicitly computable.

\begin{theorem}\label{Weakestim} Any immersed hypersurface $M\hookrightarrow\R^{n+1}$ with $r_{M}\|\H\|_2\leqslant 1+\varepsilon$ (or with $\frac{n\|\H\|_2^2}{\lambda_1^M}\leqslant1+\varepsilon$) satisfies
\begin{align}\label{pinchcm}
\bigl\||\H|-\|\H\|_2\bigr\|^{~}_2\leqslant 100 \sqrt[8]{\varepsilon}\|\H\|_2,
\end{align}
\begin{align}\label{estiray}
\Vol(M\setminus A_{\sqrt[8]{\varepsilon}})\leqslant 100\sqrt[8]{\varepsilon}v_M.\end{align}
Moreover, for any $r>0$ and any $x\in S_M=\overline{X}+\frac{1}{\|\H\|_2}\cdot\S^n$, we have
\begin{align}\label{passepartout}\Bigl|\frac{\Vol\bigl(B_x(\frac{r}{\|\H\|_2})\cap M\bigr)}{v_M}-\frac{\Vol\bigl(B_x(\frac{r}{\|\H\|_2})\cap S_M\bigr)}{\Vol S_M}\Bigr|\leqslant \tau(\varepsilon|n,r)\frac{\Vol\bigl(B_x(\frac{r}{\|\H\|_2})\cap S_M\bigr)}{\Vol S_M}.\end{align}
\end{theorem}

Note that \eqref{passepartout} implies not only that $M$ goes near any point of the sphere $S_M$, but also that the density of $M$ near each point of $S_M$ converge to $v_M/\Vol S_M$ at any scale. However, the convergence is not uniform with respect to the scales $r$. We infer that $A_{\tau(\varepsilon|n)}\cap M$ is Hausdorff close to $S_M$, which implies weak convergence to $S_M$ of almost extremal hypersurfaces.

\begin{corollary}\label{WeakHausdorff}
For any immersed hypersurface $M\hookrightarrow\R^{n+1}$ with $r_{M}\|\H\|_2\leqslant 1+\varepsilon$ (or with $\frac{n\|\H\|_2^2}{\lambda_1^M}\leqslant1+\varepsilon$) there exists a subset $A\subset M$ such that $\Vol(M\setminus A)\leqslant \tau(\varepsilon|n)v_M$ and $d_H(A,S_M)\leqslant\frac{\tau(\varepsilon|n)}{\|\H\|_2}$.
\end{corollary}

In the case where $M$ is the boundary of a convex boby in $\R^{n+1}$ with $r_{M}\|\H\|_2\leqslant 1+\varepsilon$ (or with $\frac{n\|\H\|_2^2}{\lambda_1^M}\leqslant1+\varepsilon$), the previous result implies easily that $d_H(M,S_M)\leqslant\frac{\tau(\varepsilon|n)}{\|H\|_2}$ and even $d_L(M,S_M)\leqslant\frac{\tau(\varepsilon|n)}{\|H\|_2}$.

\subsection{Hausdorff limit-set of almost extremal hypersurfaces}

Corollary \ref{WeakHausdorff} and Theorems \ref{Main} and \ref{constrfond} (applied to $M_1=\S^n$ and $M_2$ any immersible hypersurface) allow a description of the limit-set of almost extremal hypersurfaces under a priori bounds on the mean curvature.

\begin{theorem}\label{ctrexple4}
Let $M$ be any hypersurface immersible in $\R^{n+1}$ and $T$ be a closed subset of $\R^{n+1}$, such that $\S^n\cup T$ is connected (resp. and $T\cup\S^n\subset B_0(1)$). There exists a sequence of immersions $j_i:M\hookrightarrow\R^{n+1}$ of $M$ which satisfies

1) $\lambda_1^{j_i(M)}\to\lambda_1(\S^n)$ (resp. $r_{j_i(M)}\to 1$),

2) $\|\B_i-{\rm Id}\|_p\to 1$ for any $p\in[2,n-1)$,

3) $\Vol j_i(M)\to \Vol\S^n$,

4) $j_i(M)$ converges to $\S^n\cup T$ in pointed Hausdorff distance,

5) $\Vol\S^n\|\H_i\|_{n-1}^{n-1}\to C(n) m_1(T)+\Vol\S^n$.
\end{theorem}

This result shows that we can expect no control on the topology of almost extremal hypersurfaces nor on the metric shape (even on the diameter) of the part $M\setminus A$ of Corollary \ref{WeakHausdorff} if we do not assume a strong enough upper bound on the curvature. 

On the other hand, Theorem \ref{Main} implies the following Hausdorff stability result.

\begin{theorem}\label{Hausdorff}
For any immersed hypersurface $M\hookrightarrow\R^{n+1}$ with 
$v_M\|\H\|_{n-1}^n\leqslant A$ and $r_M\|\H\|_2\leq 1+\varepsilon$ (or with $v_M\|\H\|_{n-1}^n\leqslant A$ and $\frac{n\|\H\|_2^2}{\lambda_1}\leqslant1+\varepsilon$) there exists a subset $T$ of 1-dimensional Haussdorff measure less than $C(n)\int_M|\H|^{n-1}\leqslant C(n)A\|\H\|_2^{-1}$ such that $T\cup S_M$ is connected and $d_H(M,S_M\cup T)\leqslant\tau(\varepsilon|n,A)\|\H\|_2^{-1}$.

More precisely, for any sequence $(M_k)_{k\in\N}$ of immersed hypersurfaces normalized by $\|\H_k\|_2=1$ and $\overline{X}_k=0$, which satisfies
$v_{M_k}\|\H_k\|_{n-1}^n\leqslant A$ and $r_{M_k}\to 1$ (or $v_{M_k}\|\H_k\|_{n-1}^n\leqslant A$ and $\frac{n}{\lambda_1(M_k)}\to 1$) there exists a closed subset $T\subset\R^{n+1}$ such that $m_1(T)\leqslant C(n)A$, $T\cup \S^n$ is connected and a subsequence $M_{k'}$ such that $d_H(M_{k'},\S^n\cup T)\to 0$.
\end{theorem}

Here also the constant $C(n)$ of this theorem is not the same as in Theorem \ref{ctrexple4}. So we do not have an exact computation of the Hausdorff limit set in the case $p=n-1$ but we conjecture that it is just a mater of non optimality of the constant $C(m)$ in the bound on $m_1(T)$ in Theorem \ref{Main}.

Finally, as a direct consequence of Theorem \ref{Main}, we get the following result.

\begin{theorem}\label{HausdorffCritic}
Let $2\leqslant n-1<p\leqslant+\infty$. Any immersed hypersurface $M\hookrightarrow\R^{n+1}$ with 
$v_M\|\H\|_p^n\leqslant A$ and $r_M\|\H\|_2\leq 1+\varepsilon$ (or with $v_M\|\H\|_p^n\leqslant A$ and $\frac{n\|\H\|_2^2}{\lambda_1}\leqslant1+\varepsilon$) satisfies $d_H(M,S_M)\leqslant\tau(\varepsilon|n,p,A)\|\H\|_2^{-1}$.
\end{theorem}

Theorem \ref{HausdorffCritic} was already proved in the case $p=+\infty$ and under the stronger assumption $(1+\varepsilon)\lambda_1\geqslant n\|\H\|_4^2$  in \cite{colgros}, and in the case $p=+\infty$ and under the stronger assumption $r_M\|\H\|_4\leq 1+\varepsilon$ in \cite{roth}. It is also proved in an unpublished previous version of this paper \cite{AGR1} in the case $p>n$. In all these papers, the Hausdorff convergence is obtained by first proving that $\|X\|$ is almost constant in $L^2$ norm and then by applying a Moser iteration technique to infer that $\|X\|$ is almost constant is $L^\infty$-norm. However, this scheme of proof cannot be applied to get the optimal condition $p>n-1$ since $p=n$ is the critical exponent for the iteration.
In place of a Moser iteration scheme, we adapt a technique introduced by P.Topping \cite{Top} to control the diameter of $M$ by $\int_M|\H|^{n-1}\, dv$.

Note that by Theorem \ref{Hausdorff}, in the case $v_M\|\H\|^n_p\leqslant A$ with $p>n-1$, almost extremal hypersurfaces for the Reilly inequality are almost extremal hypersurfaces for the Hasanis-Koutroufiotis inequality. Actually, in that case, an hypersurface is Hausdorff close to a sphere if and only if it is almost extremal for the Hasanis-Koutroufiotis inequality. In \cite{AG1}, we prove that an hypersurface Hausdorff close to a sphere or almost extremal for the Hasanis-Koutroufiotis inequality is not necessarily almost extremal for the Reilly inequality, even under the assumption $v_M\|\H\|_p^n\leqslant A$, for any $p<n$. \medskip

The structure of the paper is as follows: in Section \ref{concentration}, we recall some concentration properties for the volume and the mean curvature of almost extremal hypersurfaces (in particular Inequalities \eqref{estiray} and \eqref{pinchcm}) and some estimates on the restrictions to hypersurfaces of the homogeneous, harmonic polynomials of $\R^{n+1}$, proved in \cite{AG1}. They are used in Section \ref{poipa} to prove Inequality \eqref{passepartout}. Theorem \ref{Main} is proved in Section \ref{Topping}. We end the paper in section \ref{se} by the proof of Theorem \ref{constrfond}.

 Throughout the paper we adopt the notation that $C(n,k,p,\cdots)$ is function greater than $1$ which depends on $p$, $q$, $n$, $\cdots$. It eases the exposition to disregard the explicit nature of these functions. The convenience of this notation is that even though $C$ might change from line to line in a calculation it still maintains these basic features.
\bigskip

\noindent\underline{Acknowledgments}: 
 We thank C.Ann\'{e} and P.Jammes for very fruitful discussions on Theorem \ref{constrfond}. Part of this work was done while E.A was invited at the MSI, ANU Canberra, funded by the PICS-CNRS Progress in Geometric Analysis and Applications. E.A. thanks P.Delanoe, J.Clutterbuck and J.X. Wang for giving him this opportunity and for stimulating discussions on that work.


\section{Some estimates on almost extremal hypersurfaces}\label{concentration}

We recall some estimate on almost extremal hypersurfaces proved in \cite{AG1}.
From now on, we assume, without loss of generality, that $\bar{X}=0$. We say that $M$ satisfies the pinching $(P_{p,\varepsilon})$ when $\|\H\|_p\|X-\overline{X}\|_2\leqslant 1+\varepsilon$. Let $X^T(x)$ denote the orthogonal projection of $X(x)$ on the tangent space $T_xM$.

\begin{lemma}[\cite{AG1}]\label{banalite} If $(P_{2,\varepsilon})$  holds, then we have $\|X^T\|_2\leqslant\sqrt{3\varepsilon}\|X\|_2$ and $\|X-\frac{\H}{\|\H\|_2^2}\nu\|_2\leqslant\sqrt{3\varepsilon}\|X\|_2$.
\end{lemma}

We set $A_\eta=B_0(\frac{1+\eta}{\|\H\|_2})\setminus B_0(\frac{1-\eta}{\|\H\|_2})$.

\begin{lemma}[\cite{AG1}]\label{estimplus}
If $(P_{p,\varepsilon})$ (for $p>2$), or $ n\|\H\|_2^2/\lambda_1^M\leqslant1+\varepsilon$, or $r_{M}\|\H\|_2\leqslant1+\varepsilon$ holds (with $\varepsilon\leqslant\frac{1}{100}$), then we have
$\bigl\|\|X\|-\frac{1}{\|\H\|_2}\bigr\|_2\leqslant\frac{C}{\|\H\|_2}\sqrt[8]{\varepsilon}$, $\||\H|-\|\H\|_2\|_2\leq C\sqrt[8]{\varepsilon} \|\H\|_2$
and
$\Vol (M\setminus A_{\sqrt[8]{\varepsilon}})\leq C\sqrt[8]{\varepsilon}v_M$,
where $C=6\times2^\frac{2p}{p-2}$ in the case ($P_{p,\varepsilon}$) and $C=100$ in the other cases. 
\end{lemma}

We set $\mathcal{H}^k(M)$ the set of functions $\{P_{|M}\}$, where $P$ is any harmonic, homogeneous polynomials of degree $k$ of $\R^{n+1}$. We also set $\psi{:}[0,\infty)\rightarrow [0,1]$ a smooth function, which is $0$ outside $[\frac{(1-2\sqrt[16]{\varepsilon})^2}{\|\H\|_2^2},\frac{(1+2\sqrt[16]{\varepsilon})^2}{\|\H\|_2^2}]$ and $1$ on $[\frac{(1-\sqrt[16]{\varepsilon})^2}{\|\H\|_2^2},\frac{(1+\sqrt[16]{\varepsilon})^2}{\|\H\|_2^2}]$, and  $\varphi$ the function on $M$ defined by $\varphi(x)=\psi(|X_x|^2)$.

\begin{lemma}[\cite{AG1}]\label{Ppresquortho2}
For any hypersurface $M\hookrightarrow\R^{n+1}$ isometrically immersed with $r_{M}\|\H\|_2\leqslant1+\varepsilon$ (or $\frac{ n\|\H\|_2^2}{\lambda_1}\leqslant1+\varepsilon$) and for any $P\in\hkm$, we have
$$\bigl|\|\H\|_2^{2k}\normdfonc{\varphi P}^2-\normdpoly{P}^2\bigr|\leqslant C\sqrt[32]{\varepsilon}\normdpoly{P}^2,$$
where $C=C(n,k)$.

If moreover $\varepsilon\leqslant\frac{1}{(2C)^{32}}$, then we have $\bigl\|\Delta(\varphi P)-\mu_k^{S_M}\varphi P\bigr\|_2\leqslant C\sqrt[16]{\varepsilon}\mu_k^{S_M}\|\varphi P\|_2$.
\end{lemma}

\section{Proof of Inequality \ref{passepartout}}\label{poipa}

By a homogeneity, we can assume $\|\H\|_2=1$. Let $\theta\in(0,1)$, $x\in\S^n$ and set $V^n(s)=\Vol (B(x,s)\cap \S^n)$. Let $\beta(\theta,r)>0$ small enough so that $(1+\theta/2)V^n\bigl((1+2\beta)r\bigr)\leqslant (1+\theta)V^n(r)$ and $(1-\theta/2)V^n\bigl((1-2\beta)r\bigr)\geqslant(1-\theta) V^n(r)$. Let $f_1:\S^n\to[0,1]$ (resp. $f_2:\S^n\to[0,1]$) be a smooth function such that $f_1=1$ on $B_x\bigl((1+\beta)r\bigr)\cap\S^n$ (resp. $f_2=1$ on $B_x\bigl((1-2\beta)r\bigr)\cap\S^n$) and $f_1=0$ outside $B_x\bigl((1+2\beta) r\bigr)\cap\S^n$ (resp. $f_2=0$ outside $B_x\bigl((1-\beta) r\bigr)\cap\S^n$). There exist an integer $N(\theta,r)$ and a family $(P^i_k)_{k\leqslant N}$ such that $P^i_k\in\hkr$ and $A=\sup_{\S^n}\bigl|f_i-\sum_{k\leqslant N}P^i_k\bigr|\leqslant \|f_i\|_{\S^n}\theta/18$. We extend $f_i$ to $\R^{n+1}\setminus\{0\}$ by $f_i(X)=f_i\bigl(\frac{X}{|X|}\bigr)$. Then we have
$$\displaylines{
\Bigl|\|\varphi f_i\|_2^2-\frac{1}{\Vol\S^n}\int_{\S^n}|f_i|^2\Bigr|\leqslant I_1+I_2+I_3}$$
where 
$$I_1:=\Bigl|\frac{1}{v_M}\int_M\Bigl(|\varphi f_i|^2-\varphi^2\bigl(\sum_{k\leqslant N}|X|^{-k}P^i_k\bigr)^2\Bigr)\vol\Bigr|$$
$$I_2:=\Bigl|\frac{1}{v_M}\int_M\varphi^2\bigl(\sum_{k\leqslant N}|X|^{-k}P^i_k\bigr)^2\vol-\sum_{k\leqslant N}\|P^i_k\|_{\S^n}^2\Bigr|$$
and
$$I_3:=\Bigl|\frac{1}{\Vol\S^n}\int_{\S^n}\Bigl(\bigl(\sum_{k\leqslant N}P^i_k\bigr)^2-f_i^2\Bigr)\Bigr|.$$
On $\S^n$ we have $\bigl|f_i^2-(\sum_{k\leqslant N}P^i_k)^2\bigr|\leqslant A\bigl(2\sup_{\S^n}|f_i|+A\bigr)\leqslant\|f_i\|_{\S^n}^2\theta/6$ and on $M$ we have
\begin{align*}\varphi^2\Bigl|f_i^2(X)-\bigl(\sum_{k\leqslant N}|X|^{-k}P^i_k(X)\bigr)^2\Bigr|&\leqslant\Bigl|f_i^2\bigl(\frac{X}{|X|}\bigr)-\Bigl(\sum_{k\leqslant N}P^i_k\bigl(\frac{X}{|X|}\bigr)\Bigr)^2\Bigr|\leqslant\|f_i\|_{\S^n}^2\theta/6\cr.
\end{align*}
Hence $I_1+I_3\leqslant\|f_i\|_{\S^n}^2\theta/3$. Now
\begin{align*}
I_2\leqslant&\Bigl|\frac{1}{v_M}\insm\varphi^2\sum_{k\leqslant N}\frac{(P_k^i)^2}{|X|^{2k}}\vol-\sum_{k\leqslant N}\|P^i_k\|_{\S^n}^2\Bigr|+\frac{1}{v_M}\Bigl|\insm\varphi^2\sum_{1\leqslant k\neq k'\leqslant N}\frac{P_k^i P_{k'}^i}{|X|^{k+k'}}\vol\Bigr|\\
\leqslant&\frac{1}{v_M}\insm\varphi^2\sum_{k\leqslant N}\Bigl|\frac{1}{|X|^{2k}}-\|\H\|_2^{2k}\Bigr|(P_k^i)^2\vol \\
&+\frac{1}{v_M}\insm\sum_{1\leqslant k\neq k'\leqslant N}\varphi^2\Bigl|\frac{1}{|X|^{k+k'}}-\|\H\|_2^{k+k'}\Bigr||P_k^i P_{k'}^i|\vol\\
&+\sum_{k\leqslant N}\bigl|\|\H\|_2^{2k}\normdfonc{\varphi P_k^i}^2-\normdpoly{P_k^i}^2\bigr|+\sum_{1\leqslant k\neq k'\leqslant N}\frac{\|\H\|_2^{k+k'}}{v_M}\Bigl|\insm\varphi^2 P_k^i P_{k'}^i\vol\Bigr|
\end{align*}
We have $\varphi^2\bigl|\frac{1}{|X|^{k+k'}}-\|\H\|_2^{k+k'}\bigr|\leqslant \varphi^2 (k+k')2^{k+k'+2}\sqrt[16]{\varepsilon}\|\H\|_2^{k+k'}$ by assumption on $\varphi$. From this and Lemma \ref{Ppresquortho2}, we have
\begin{align*}I_2\leqslant& N^24^{N+1}\sqrt[16]{\varepsilon}\sum_{k\leqslant N}\|\H\|_2^{2k}\normdfonc{\varphi P_k^i}^2+\sqrt[32]{\varepsilon}\sum_{k\leqslant N}C(n,k)\normdpoly{P_k^i}^2\\
&+\sum_{1\leqslant k\neq k'\leqslant N}\frac{\|\H\|_2^{k+k'}}{v_M}\Bigl|\insm \varphi^2 P_k^i P_{k'}^i\vol\Bigr|\\
\leqslant& C(n,N)\sqrt[32]{\varepsilon}+\sum_{1\leqslant k\neq k'\leqslant N}\frac{\|\H\|_2^{k+k'}}{v_M}\Bigl|\insm \varphi^2 P_k^i P_{k'}^i\vol\Bigr|
\end{align*}
and, by Lemma \ref{Ppresquortho2}, we have
\begin{align*}
&|\frac{\|\H\|_2^2(\mu_k-\mu_{k'})}{v_M}\int_M\varphi^2 P^i_kP^i_{k'}\vol\Bigr|\\
&\leqslant\int_M\frac{|\varphi P^i_k\bigl(\Delta(\varphi P^i_{k'})-\|\H\|_2^2\mu_{k'}\varphi P^i_{k'}\bigr)|}{v_M}\vol+\int_M\frac{|\varphi P^i_{k'}\bigl(\Delta(\varphi P^i_k)-\|\H\|_2^2\mu_k\varphi P^i_k\bigr)|}{v_M}\vol\\
&\leqslant \|\varphi P^i_k\|_2\bigl\|\Delta(\varphi P^i_{k'})-\|\H\|_2^2\mu_{k'}\varphi P^i_{k'}\bigr\|_2+\|\varphi P^i_{k'}\|_2\bigl\|\Delta(\varphi P^i_k)-\|\H\|_2^2\mu_k\varphi P^i_k\bigr\|_2\\
&\leqslant C(n,N)\sqrt[16]{\varepsilon}\|\H\|_2^2\|\varphi P^i_{k'}\|_2\|\varphi P^i_k\|_2
\end{align*}
under the condition $\varepsilon\leqslant(\frac{1}{2C(n,N)})^{32}$. Since $\mu_k-\mu_{k'}\geqslant n$ when $k\neq k'$, we get
$$\sum_{1\leqslant k\neq k'\leqslant N}\Bigl|\frac{1}{v_M}\int_M\varphi^2 P^i_kP^i_{k'}\vol\Bigr|\leqslant \sum_{1\leqslant k\neq k'\leqslant N}C(n,N)\sqrt[16]{\varepsilon}\|\varphi P^i_{k'}\|_2\|\varphi P^i_k\|_2\leqslant\frac{C(n,N)\sqrt[16]{\varepsilon}}{\|\H\|_2^{k+k'}}$$
hence $I_2\leqslant C(n,N)\sqrt[32]{\varepsilon}$ and
$$\Bigl|\|\varphi f_i\|_2^2-\frac{1}{\Vol\S^n}\int_{\S^n}f_i^2\Bigr|\leqslant C(n,N)\sqrt[32]{\varepsilon}+\frac{\theta}{3}\|f_i\|_{\S^n}^2.$$
We infer that if $\sqrt[32]{\varepsilon}\leqslant\frac{V^n((1-2\beta)r)\theta}{6C(n,N)\Vol\S^n}\leqslant\frac{\|f_i\|_{\S^n}^2\theta}{6C(n,N)}$, then we have
$$\Bigl|\|\varphi f_i\|_2^2-\frac{1}{\Vol\S^n}\int_{\S^n}|f_i|^2\Bigr|\leqslant\theta\|f_i\|^2_{\S^n}/2$$
Note that $N$ depends on $r$ and $\theta$ but not on $x$ since $O(n+1)$ acts transitively on $\S^n$.
By assumption on $f_1$ and $f_2$, we have
\begin{align*}
&&\frac{\Vol( B_x((1+\beta)r-\sqrt[16]{\varepsilon}))\cap M\cap A_{\sqrt[16]{\varepsilon}})}{v_M}&\leqslant\|\varphi f_1\|_2^2\leqslant(1+\frac{\theta}{2})\|f_1\|_{\S^n}^2\\
&&&\leqslant(1+\frac{\theta}{2})\frac{V^n((1+2\beta) r)}{\Vol\S^n}\leqslant(1+\theta)\frac{V^n( r)}{\Vol\S^n}\\
&&\frac{\Vol( B_x((1-\beta)r+2\sqrt[16]{\varepsilon})\cap M\cap A_{2\sqrt[16]{\varepsilon}})}{v_M}&\geqslant\|\varphi f_2\|_2^2\geqslant(1-\frac{\theta}{2})\|f_2\|_{\S^n}^2\\
&&&\geqslant(1-\frac{\theta}{2})\frac{V^n((1-2\beta) r)}{\Vol\S^n}\geqslant(1-\theta)\frac{V^n( r)}{\Vol\S^n}
\end{align*}
In the second estimates, we can replace $\varepsilon$ by $\varepsilon/2^{16}$ as soon as we assume that $\varepsilon\leqslant\bigl(\min(\frac{1}{4^{16}},\frac{1}{(2C(n,N))^{32}},(\beta r)^{16},(\frac{\|f_i\|_{\S^n}^2\theta}{6(C(n,N)})^{32}\bigr)=K(\theta,r,n)$. Then we have $(1-\beta)r+\sqrt[16]{\varepsilon}\leqslant r\leqslant (1+\beta)r-\sqrt[16]{\varepsilon}$ and get
$$\Bigl|\frac{\Vol(B_x(r)\cap M\cap A_{\sqrt[16]{\varepsilon}})}{v_M}-\frac{V^n(r)}{\Vol \S^n}\Bigr|\leqslant\theta\frac{V^n(r)}{\Vol \S^n}$$
Combined with Lemma \ref{estimplus}, we get the result with $\tau(\varepsilon|r,n)=\min\{\theta/\,2^{16}\varepsilon\leqslant K(\theta,r,n)\}$.


\section{Proof of Theorem \ref{Main}}\label{Topping}

In this section, we extend the technique developed by P.Topping in \cite{Top} to get an upper bound of ${\rm Diam} (M)$ by $\int_M|\H|^{n-1}$. 

\subsection{Decomposition lemma} We begin by a general result on approximation of Euclidean submanifolds in Hausdorff distance by the union of a subset of large volume and a finite family of geodesic subtrees of total length bounded by $\int_M|\H|^{n-1}$.

\begin{lemma}\label{decomp}
Let $M^m$ be an Euclidean compact submanifold of $\R^{n+1}$ and $A\subset M$ a closed subset. There exists a constant $C(m)$ and a finite family of geodesic trees $T_i\subset M$ such that $A\cap T_i\neq\emptyset$, $d_H\bigl(A\cup(\cup_i T_i),M\bigr)\leqslant C \bigl(\Vol(M\setminus A)\bigr)^\frac{1}{m}$ and $\sum_im_1(T_i)\leqslant C^{m(m-1)}\int_{M\setminus A}|\H|^{m-1}$.
\end{lemma}

\begin{proof}
In \cite{Top}, using the Michael-Simon Sobolev inequality as a differential inequation on the volume of intrinsic spheres, P.Topping prove the following lemma (slightly modified for our purpose). 

\begin{lemma}[\cite{Top}]\label{Top}
Suppose that $M^m$ is a submanifold smoothly
immersed in $\R^{n+1}$, which is complete with respect to the induced metric. Then there
exists a constant $\delta(m)>0$ such that for any $x\in M$ and $R>0$, at
least one of the following is true:
\begin{itemize}
\item[(i)] $M(x,R):=\sup_{r\in(0,R]}\int_{B_x(r)}|\H|^{m-1}/r> \delta^{m-1}$;
\item[(ii)] $\kappa(x,R):=\inf_{r\in(0,R]}\frac{\Vol B_x(r)}{r^m}>\delta$.
\end{itemize}
Where $B_x(r)$ is the geodesic ball in $M$ for the intrinsic distance.
\end{lemma}

In this section, $d$ stands for the intrinsic distance on $M$.

If $d_H(A,M)\leqslant 10(\frac{\Vol M\setminus A}{\delta(m)})^\frac{1}{m} $, then we set $T=\emptyset$.

Otherwise, there exists $x_0\in M$ such that $d(A,x_0)=d_H(A,M)\geqslant 10(\frac{\Vol M\setminus A}{\delta(m)})^\frac{1}{m} $.
Let $\gamma_{0}:[0,l_0]\to M\setminus A$ be a normal minimizing geodesic from $x_{0}$ to $A$. For any $t\in I_0=[0,l_0-(\frac{\Vol M\setminus A}{\delta(m)})^\frac{1}{m}]$, we have $B_{\gamma_{0}(t)}\bigl((\frac{\Vol M\setminus A}{\delta(m)})^\frac{1}{m}\bigr)\subset M\setminus A$ and by the previous lemma, there exists $r_{0,t}\leqslant(\frac{\Vol M\setminus A}{\delta(m)})^\frac{1}{m}$ such that $r_{0,t}\leqslant\frac{1}{\delta^{m-1}}\int_{B_{\gamma_{0}(t)}(r_{0,t})}|\H|^{m-1}$. By compactness of $\gamma_{0}(I_0)$ and by Wiener's selection principle, there exists a finite family $(t_j)_{j\in J_0}$ of elements of $I_0$ such that the balls of the family $\mathcal{F}_0=\bigl(B_{\gamma_{0}(t_j)}(r_{0,t_j})\bigr)_{j\in J_0}$ are disjoint and $\gamma(I_0)\subset \cup_{j\in J_0}B_{\gamma_{0}(t_j)}(3r_{0,t_j})$. Hence we have
\[\frac{\delta^{m-1} (l_0-(\frac{\Vol M\setminus A}{\delta})^\frac{1}{m})}{6}\leqslant\delta^{m-1}\sum_{j\in J_0}r_{0,t_j}\leqslant\sum_{j\in J_0}\int_{B_{\gamma_{0}(t_j)}(r_{0,t_j})}|\H|^{m-1}\]
And by assumption on $l_0$, we get $10(\frac{\Vol M\setminus A}{\delta(m)})^\frac{1}{m} \leqslant l_0\leqslant\frac{10}{\delta^{m-1}}\sum_{j\in J_0}\int_{B_{\gamma_{0}(t_j)}(r_{0,t_j})}|\H|^{m-1}$.

If $d_H\bigl(A\cup\gamma_{0}([0,l_0]),M\bigr)\leqslant 10(\frac{\Vol M\setminus A}{\delta(m)})^\frac{1}{m}$, we set $T=\gamma_{0}([0,l_0])$. Otherwise, we set $x_{1}$ a point of $M\setminus A$ at maximal distance $l_1$ from $A\cup\gamma_{0}([0,l_0])$ and $\gamma_{1}$ the corresponding minimal geodesic. We set $I_1=[2(\frac{\Vol M\setminus A}{\delta(m)})^\frac{1}{m},l_1-2(\frac{\Vol M\setminus A}{\delta(m)})^\frac{1}{m}]$. Once again, by the Wiener Lemma applied to $\gamma_{1}(I_1)$ we get a family of disjoint balls $\mathcal{F}_1=\bigl(B_{\gamma_{1}(t_j)}(r_{1,t_j}))_{j\in J_1}$ such that
$$\frac{\delta^{m-1}(l_1-4(\frac{\Vol M\setminus A}{\delta})^\frac{1}{\delta})}{6}\leqslant\delta^{m-1}
\sum_{j\in J_1}r_{1,t_j}\leqslant\sum_{j\in J_1}\int_{B_{\gamma_{1}(t_j)}(r_{1,t_j}^{~})}|\H|^{m-1}$$
which gives $10(\frac{\Vol M\setminus A}{\delta(m)})^\frac{1}{m} \leqslant l_1\leqslant\frac{10}{\delta^{m-1}}\sum_{j\in J_1}\int_{B_{\gamma_{1}(t_j)}(r_{1,t_j}^{~})}|\H|^{m-1}$.
Note also that the balls of the family $\mathcal{F}_1\cup\mathcal{F}_2$ are disjoint.

If $d_H\bigl(A\cup\gamma_{0}([0,l_0])\cup\gamma_{1}([0,l_1]),M\bigr)\leqslant10(\frac{\Vol M\setminus A}{\delta(m)})^\frac{1}{m}$, we set $T=\gamma_{0}([0,l_0])\cup\gamma_{1}([0,l_1])$. Note that $T$ is a geodesic tree (if $\gamma_{1}(l_1)\in\gamma_{0}([0,l_1])$) or the disjoint union of 2 geodesic trees.

If $d_H\bigl(A\cup\gamma_{0}([0,l_0])\cup\gamma_{1}([0,l_1]),M\bigr)\geqslant10(\frac{\Vol M\setminus A}{\delta(m)})^\frac{1}{m}$, then by iteration of what was made for $x_{1}$, $\gamma_{1}$ and $\mathcal{F}_1$, we construct a family $(x_{j})_j$ of points, a family  $(\gamma_{j})_j$ of geodesics and a family $(\mathcal{F}_j)_j$ of sets of disjoint balls. Since the $(x_{j})_j$ are $10(\frac{\Vol M\setminus A}{\delta(m)})^\frac{1}{m}$-separated in $M$ and since $M$ is compact, the families are finite and only a finite step of iterations can be made. The set $T=\cup_j\gamma_{j}([0,l_j])$ is the disjoint union of a finite set of finite geodesic trees and we have
\begin{align}\label{length}
 l(T)\leqslant\frac{10}{\delta^{m-1}}\sum_j\sum_{k\in J_j}\int_{B_{\gamma_{j}(t_j)}(r_{j,t_k})}|\H|^{m-1}\leqslant\frac{10}{\delta^{m-1}}\int_{M\setminus A}|\H|^{m-1}.
\end{align}
Moreover, the connected parts of $T$ are geodesic trees whose number is bounded above by $\frac{\delta^{\frac{m+1-m^2}{m}}}{(\Vol M\setminus A)^\frac{1}{m}}\int_{M\setminus A}|\H|^{m-1}$.
\end{proof}

\subsection{Proof of Theorem \ref{Main}}
We begin the proof by the case where $\int_{M_k}|\H|^{m-1}\leqslant A$. By Topping's upper bound on the diameter \cite{Top} and Blaschke selection theorem, the sequence $M_k$ converges, in Hausdorff topology, to a closed, connected limit set $M_\infty$, which contains $Z$.

It just remain to prove that $m_1(M_\infty\setminus Z)\leqslant C(m)A$. Let $\ell\in\N^*$ fixed. We set $Z_r=\{x\in\R^{n+1}/\,d(x,Z)\leqslant r\}$. By the Michael-Simon Sobolev inequality applied to a constant function, we have $(\Vol M_k)^\frac{1}{m}\leqslant C(m)\int_{M_k}|\H|^{m-1}$. By weak convergence of $(M_k)_k$ to $Z$, we have $\lim_k\Vol (M_k\setminus Z_{1/2\ell})/\Vol(M_k)=0$ and by Lemma \ref{decomp}, there exists a finite union of geodesic trees $T_k^\ell$ such that $\lim_kd_H\bigl((M_k\cap Z_{1/3\ell})\cup T_k^\ell,M_\infty\bigr)=0$ and $l(T_k^\ell)\leqslant C(m)\int_{M_k\setminus Z_{1/3\ell}}|\H|^{m-1}$ for any $k$. Moreover, by construction of the part $T$ in the proof of Lemma \ref{decomp}, each connected part of $T_k^\ell$ is a geodesic tree intersecting $Z_{1/3\ell}$, and by Inequality \eqref{length}, the number of such component leaving $Z_{2/3\ell}$ is bounded above by $3\ell C(m)\int_{M_k\setminus Z_{1/3\ell}}|\H|^{m-1}$. We can assume that this number is constant up to a subsequence. Their union forms a sequence of compact sets $(\tilde{T}^\ell_k)$ which, up to a subsequence, converges to a set $Y$ that contains $M_\infty\setminus Z_{1/\ell}$. By lower semi-continuity of the $m_1$-measure for sequence of trees (see Theorem 3.18 in \cite{Fa}), we get that $m_1(M_\infty\setminus Z_{1/\ell})\leqslant\liminf_kl(\tilde{T}^\ell_k)\leqslant C(m)\liminf_k\int_{M_k\setminus Z_{1/3\ell}}|\H|^{m-1}$. Since $M_\infty\setminus Z=\cup_{\ell\in\N^*}M_\infty\setminus Z_{1/\ell}$, we get the result. So $M_\infty=Z\cup T$ with $T$ a $1$-dimensional subset of $\R^{n+1}$ of measure less than $C(m)\liminf_k\int_{M_k\setminus Z}|\H|^{m-1}\leqslant C(m)A$.

In the case $\int_{M_k}|\H|^p\leqslant A$ with $p>m-1$, we have 
$$\int_{M_k\setminus Z_{1/3\ell}}\hskip-5mm|\H|^{m-1}\leqslant\bigl(\frac{\Vol M_k\setminus Z_{1/3\ell}}{\Vol M_k}\bigr)^\frac{p-m+1}{p}\Vol M_k\|\H\|_p^{m-1}$$ 
So the weak convergence to $Z$ implies that $m_1(M_\infty\setminus Z_{1/3\ell})=0$ for any $\ell$. Since $M_\infty\setminus Z_{1/3\ell}\neq\emptyset$ implies $m_1(M_\infty\setminus Z)\geqslant1/3\ell$  by what precedes, we get $M_\infty \subset Z$, hence $M_\infty=Z$.
\medskip

For the proof of Theorems \ref{Hausdorff} and \ref{HausdorffCritic}, we can assume that $\overline{X}(M_k)=0$ and $\|\H\|_2=1\leqslant\|\H\|_p$ by scaling. Hence we have $v_{M_k}\|\H\|_p^{n-1}\leqslant v_{M_k}\|\H\|_p^n\leqslant A$ and 
$S_{M_k}=\S^n$ for any $k$. Inequality \eqref{estiray} and Lemma \ref{decomp} give the Theorems.


\section{Proof of Theorem \ref{constrfond}}\label{se}

We first prove a weak version of Theorem \ref{constrfond}, where the set $T$ is a segment $[x_0,x_0+l\nu]$ with $x_0\in M_1$ and $\nu$ a normal vector to $M_1$ at $x_0$.

\noindent\underline{Adding of a segment and estimates on the curvature}

We take off a small ball of $M_2$ and  instead we glue smoothly a curved cylinder along one of its boundary and which is isometric to the product $[0,1]\times\frac{1}{10}\S^{m-1}$ at the neighbourhood of its other boundary component. 
\begin{center}
\includegraphics[width=1.5cm]{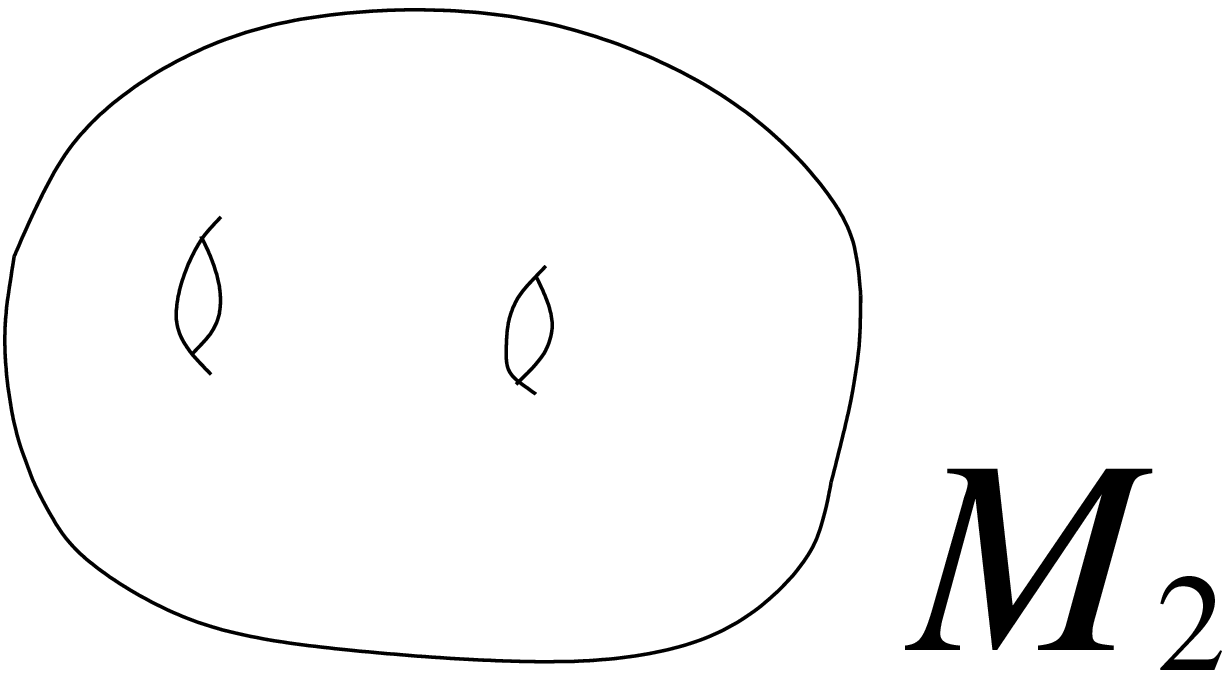}
\end{center}
We note $H_1$ the resulting submanifold and $H_\varepsilon=\varepsilon H_1$. Let $c:[0,l]\to\R^+$ be a $\mathcal{C}^1$ positive function, constant equal to $\frac{1}{10}$ at the neighbourhoods of $0$ and $l$, $T_{c,\varepsilon}$ be a cylinder of revolution isometric to $\{(t,u)\in[0,l]\times \R^{m}/|u|=\varepsilon c(t)\}$ and $J_1$ be a cylinder of revolution isometric to $[0,1/4]\times\frac{1}{10}\S^{m-1}$ at the neighbourhood of one of its boundary component and isometric to the flat annulus $B_0(\frac{3}{10})\setminus B_0(\frac{2}{10})\subset\R^m$) at the neighbourhood of its other boundary component. Note that in this paper we will only use the case $c\equiv\frac{1}{10}$ but the general case will be used in \cite{AG1}. We also set $J_\varepsilon= \varepsilon J_1$ and $N_{c,\varepsilon}$ the submanifold obtained by gluing $H_\varepsilon$, $T_{c,\varepsilon}$ and $J_\varepsilon$.
\begin{center}
\includegraphics[width=3cm]{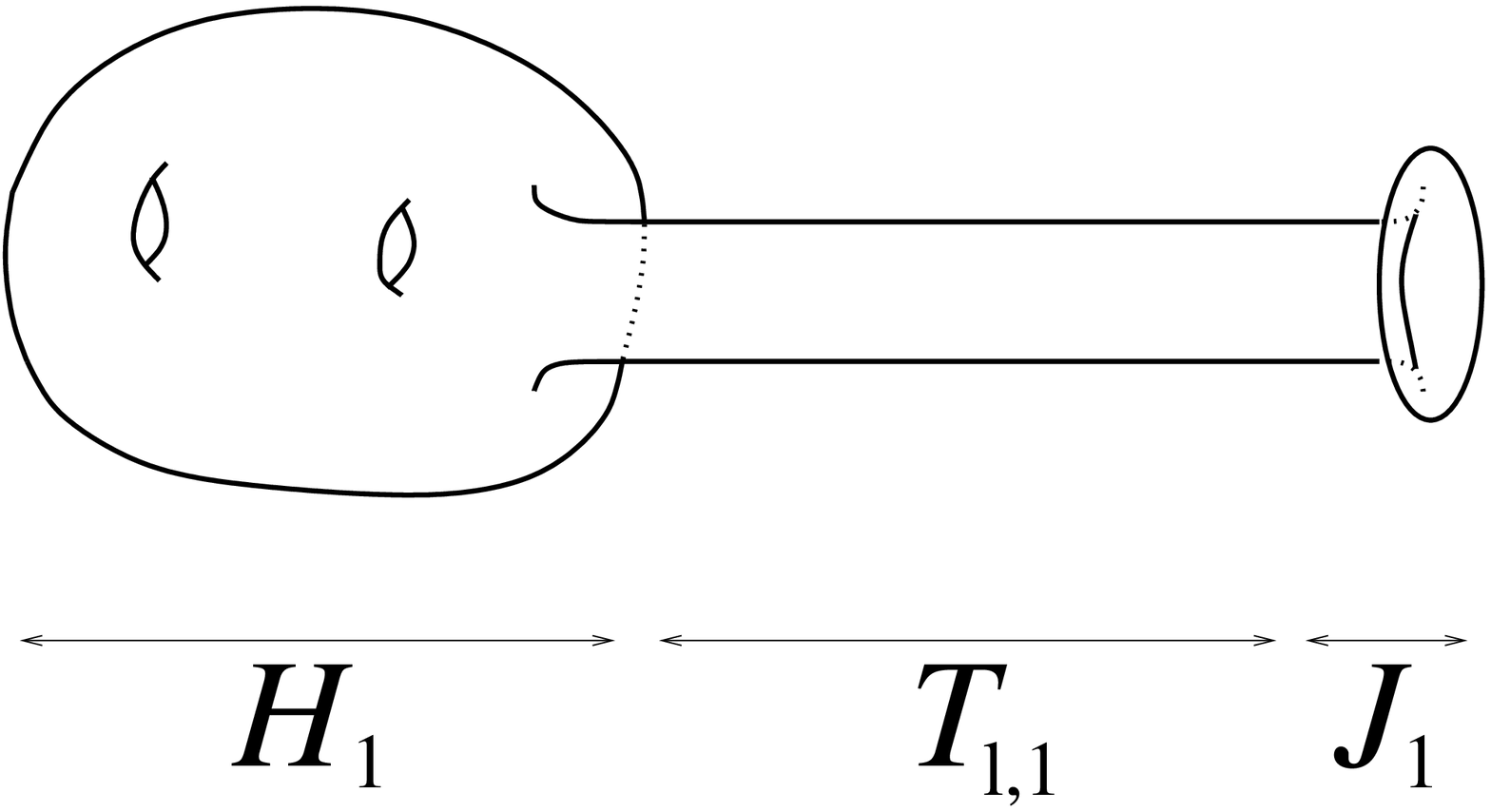}
\end{center}
Since the second fundamental form of $T_{c,\varepsilon}$ is given by $|B|^2=\frac{(\varepsilon c'')^2}{(1+(\varepsilon c')^2)^3}+\frac{m-1}{\varepsilon^2 c^2(1+(\varepsilon c')^2)}$, we have $\int_{N_{c,\varepsilon}}|\B|^\alpha dv=a(H_1,J_1)\varepsilon^{m-\alpha}+\Vol\S^{m-1}\varepsilon^{m-1-\alpha}(m-1)^\frac{\alpha}{2}\int_0^lc^{m-1-\alpha}+O_{c,\alpha}(\varepsilon^{m+1-\alpha})$, where $a(H_1,J_1)$ is a constant that depends only on $H_1$ and $J_1$ (not on $c$, $l$ and $\varepsilon$).

We set $M_1^\varepsilon$ the submanifold of $\R^{n+1}$ obtained by flattening $M_1$ at the neighbourhood of a point $x_0\in M_1$ and taking out a ball centred at $x_0$ and of radius $\frac{3\varepsilon}{10}$: $M_1$ is locally equal to $\{x_0+w+f(w),\,w\in B_0(\varepsilon_0)\subset T_{x_0}M_1\}$ where $f:B_0(\varepsilon_0)\subset T_{x_0}M_1\to N_{x_0}M_1$ is a smooth function and $N_{x_0}M_1$ is the normal bundle $M_1$ at $x_0$. Let $\varphi:\R_+\to[0,1]$ be a smooth function such that $\varphi=0$ on $[0,\frac{\varepsilon_0}{3}]$ and $\varphi=1$ on $[\frac{2\varepsilon_0}{3},+\infty)$. We set $M_1^\varepsilon$ the submanifold obtained by replacing the subset $\{x_0+w+f(w),\,w\in B_0(\varepsilon_0)\subset T_{x_0}M_1\}$ by $\{x_0+w+f_\varepsilon(w),\,w\in B_0(\varepsilon_0)\setminus B_0(3\varepsilon/10)\subset T_{x_0}M_1\}$, with $f_\varepsilon(w)=f\bigl(\varphi(\frac{\varepsilon_0\|w\|}{\varepsilon})w\bigr)$ for any $\varepsilon\leqslant3\varepsilon_0/2$. Note that $M_1^\varepsilon$ is a smooth deformation of $M_1$ in a neighbourhood of $x_0$ and its boundary has a neighbourhood isometric to a flat annulus $B_0(\varepsilon/3)\setminus B_0(3\varepsilon/10)$ in $\R^m$. Note that for $\epsilon$ small enough, $M_1^\varepsilon\setminus\{x\in M_1^\varepsilon/d(x,\partial M_1^\varepsilon)\leqslant 8\varepsilon\}$ is a subset of $M_1$. This fact will be used below. As a graph of a function, the curvatures of $M_1^\varepsilon$ at the neighbourhood of $x_0$ are given by the formulae
\begin{align*}
|\B_\varepsilon|^2&=\sum_{i,j,k,l=1}^m\sum_{p,q=m+1}^{n+1}Ddf_p(e_i,e_k)Ddf_q(e_j,e_l)H^{i,j}H^{k,l}G^{p,q}\\
\H_\varepsilon&=\frac{1}{m}\sum_{k,l=m+1}^{n+1}\sum_{i,j=1}^mDdf_k(e_i,e_j)H^{i,j}G^{k,l}(\nabla f_l-e_l)
\end{align*}
where $(e_1,\cdots,e_m)$ is an ONB of $T_{x_0}M_1$, $(e_{m+1},\cdots,e_{n+1})$ an ONB of $N_{x_0}M_1$, $f_\varepsilon(w)=\sum_{i=m+1}^{n+1}f_i(w)e_i$, $G_{kl}=\delta_{kl}+\langle\nabla f_k,\nabla f_l\rangle$ and $H_{kl}=\delta_{kl}+\langle df_\varepsilon(e_k),df_\varepsilon(e_l)\rangle$.
Now $f_\varepsilon$ converges in $\mathcal{C}^\infty$ norm to $f$ on any compact subset of $B_0(\varepsilon_0)\setminus\{0\}$, while $|df_\varepsilon|$ and $|Ddf_\varepsilon|$ remain uniformly bounded on $B_0(\varepsilon_0)$ when $\varepsilon$ tends to $0$. By the Lebesgue convergence theorem, we have
\begin{align*}
\int_{M_1^\varepsilon}|\H_\varepsilon|^{\alpha}dv\to \int_{M_1}|\H|^{\alpha}dv\quad\quad\quad \int_{M_1^\varepsilon}|\B_\varepsilon|^{\alpha}dv\to \int_{M_1}|\B|^{\alpha}dv\
\end{align*}
We set $M_\varepsilon$ the $m$-submanifold of $\R^{n+1}$ obtained by gluing $M_1^\varepsilon$ and $N_{c,\varepsilon}$ along their boundaries in a fixed direction $\nu\in N_{x_0}M_1$. Note that $M_\varepsilon$ is a smooth immersion of $M_1\#M_2$.
\begin{center}
\includegraphics[width=4cm]{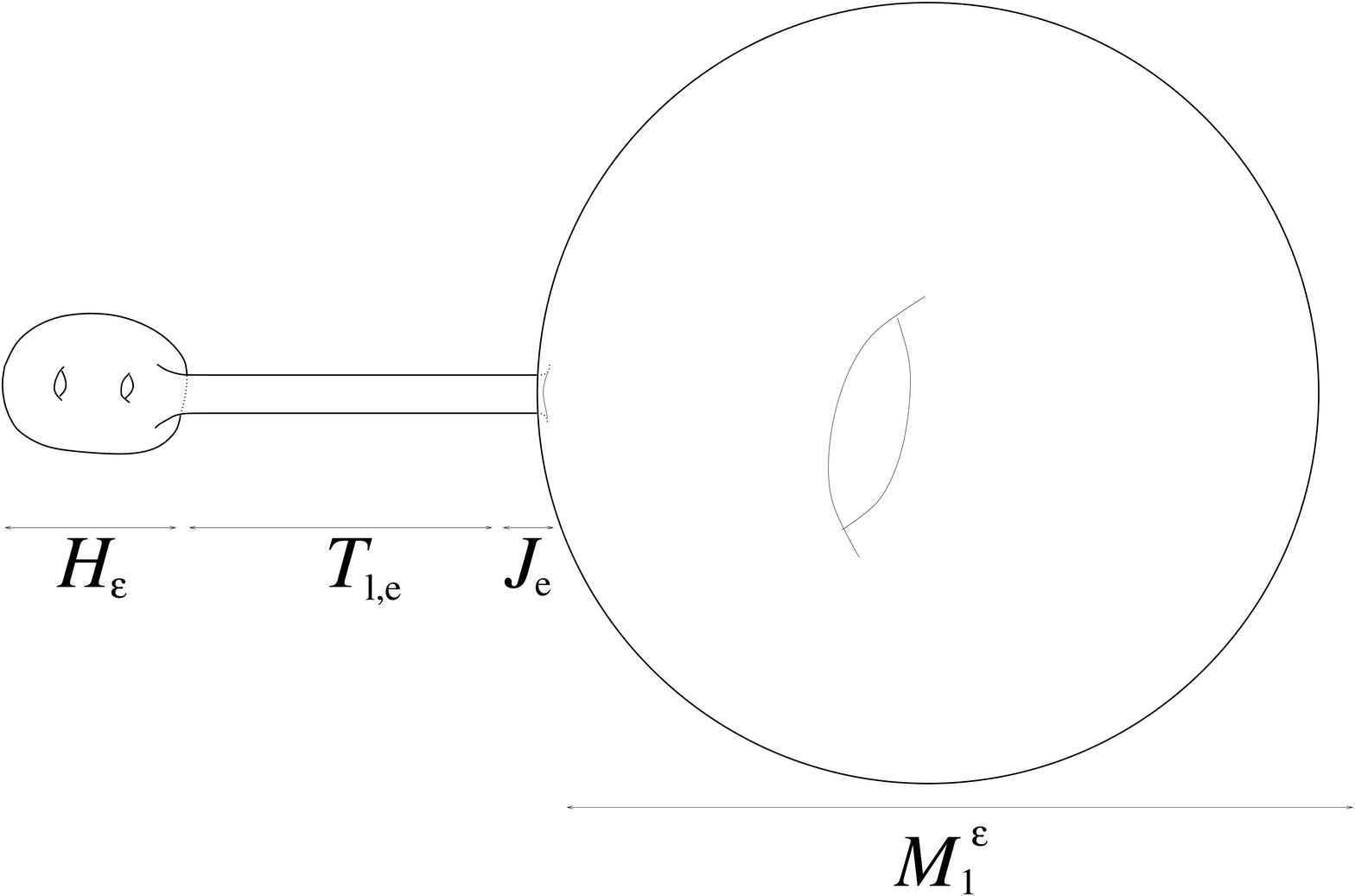}
\end{center}
By the computations above, we get the announced limits 1), 2) and 4) for the sequence $i_k(M_1\#M_2)=M_\frac{1}{k}$ as $k$ tends to $\infty$.

\noindent\underline{Computation of the spectrum}

We will adapt the method developed by C.Ann\'{e} in \cite{Ann}. We set $(\lambda_k)_{k\in\N}$ the spectrum with multiplicities obtained by union the spectrum of $M_1$ and of the spectrum $Sp(P_c)$ of the operator $P(f)=-f''-(m-1)\frac{c'}{c}f'$ on $[0,l]$ with Dirichlet condition at $0$ and Neumann condition at $l$. We denote by $(\mu_k)_{k\in\N}$ the eigenvalues of $M_1$ counted with multiplicities and by $(P_k)_{k\in\N}$ a $L^2$-ONB of eigenfunctions of $M_1$. We set $(\nu_k,h_k)_{k \in\N}$ and $(\lambda_k^\varepsilon,f_k^{\varepsilon})_{k\in\N}$ the corresponding data on $([0,l],c^{n-1}(t)\,dt)$ and $M_\varepsilon$.  We set $\tilde{h}_k^{\varepsilon}$ the function on $M_\varepsilon$ obtained by considering $h_k$ as a function on the cylinder $T_{c,\varepsilon}$, extending it continuously by $0$ on $J_\varepsilon$ and $M_1^\varepsilon$, and by $h_k(l)$ on $H_\varepsilon$. We also set $\tilde{P}_k^{\varepsilon}$ the function on $M_\varepsilon$ which is equal to $\psi_\varepsilon\bigl(d(\partial M_1^\varepsilon,\cdot)\bigr) P_k$ on $M_1^\varepsilon$ (with $\psi_\varepsilon(t)=0$ when $t\leqslant8\varepsilon$, $\psi_\varepsilon(t)=\frac{\ln t-\ln (8\varepsilon)}{-\ln(8\sqrt{\varepsilon})}$ when $t\in[8\varepsilon,\sqrt{\varepsilon}]$ and $\psi_\varepsilon(t)=1$ otherwise) and is extended by $0$ outside $M_1^\varepsilon$. Using the family $(\tilde{h}_k^\varepsilon,\tilde{P}_k^\varepsilon)$ as test functions, the min-max principle easily gives us
\begin{equation}\label{maj}
\lambda_k^\varepsilon\leqslant\lambda_k\bigl(1+\tau(\varepsilon|k,n,c, M_1)\bigr)
\end{equation}

For any $k\in\N$, we set $\alpha_k=\liminf_{\varepsilon\to0} \lambda_k^\varepsilon$, $\varphi^{(1)}_{k,\varepsilon}(x)=\varepsilon^{\frac{m}{2}}(f_k^{\varepsilon})_{|H_\varepsilon\cup J_\varepsilon}(\varepsilon x)$, seen as a function on $H_1\cup J_1$, $\varphi^{(2)}_{k,\varepsilon}(t,x)=\varepsilon^{\frac{m-1}{2}}(f_k^{\varepsilon})_{|T_{c,\varepsilon}}^{~}(t,\varepsilon c(t)x)$ seen as a function on $[0,l]\times\S^{m-1}$ and $\varphi^{(3)}_{k,\varepsilon}$ the function on $M_1$ equal to $f_k^\varepsilon$ on $\{x\in M_1^\varepsilon/\,d(x,\partial M_1^\varepsilon)\geqslant8\varepsilon\}$ and extended harmonically to $M_1$.

Easy computations give us 
\begin{align}
&\int_{H_1\cup J_1}|\varphi_{k,\varepsilon}^{(1)}|^2=\int_{H_\varepsilon\cup J_\varepsilon}|f_k^\varepsilon|^2,\ \ \ \  \int_{H_1\cup J_1}|d\varphi_{k,\varepsilon}^{(1)}|^2=\varepsilon^2\int_{H_\varepsilon\cup J_\varepsilon}|df_k^\varepsilon|^2\label{neg}\\
&\int_{T_{c,\varepsilon}}|f_k^\varepsilon|^2=\int_0^l\bigl(\int_{\S^{m-1}}|\varphi_{k,\varepsilon}^{(2)}(t,u)|^2du\bigr)\sqrt{1+\varepsilon^2(c'(t))^2}c^{m-1}(t)dt,\label{T}\\
&\int_{T_{c,\varepsilon}}|df_k^\varepsilon|^2= \int_0^l\Bigl[\frac{c^{m-1}}{\sqrt{1+\varepsilon^2(c')^2}}\int_{\S^{m-1}}|\frac{\partial\varphi_{k,\varepsilon}^{(2)}}{\partial t}|^2+\frac{\sqrt{1+\varepsilon^2(c')^2}c^{m-1}}{\varepsilon^2c^2}\int_{\S^{m-1}}|d_{\S^{m-1}}\varphi_{k,\varepsilon}^{(2)}|^2\Bigr].\label{T'}
\end{align}

The argument of C. Anne in \cite{Ann} (or of Rauch and Taylor in \cite{RT}) can be adapted to get that there exists a constant $C(M_1)$ such that $\|\varphi_{k,\varepsilon}^{(3)}\|_{H^1(M_1)}\leqslant C\|f_k^\varepsilon\|_{H^1(M_\varepsilon)}$.
Since we have 
$\|f_k^\varepsilon\|_{H^1(M_\varepsilon)}=1+\lambda_k^\varepsilon$, \eqref{maj} gives us $
\|\varphi_{k,\varepsilon}^{(3)}\|_{H^1(M_1)}\leqslant C(k,M_1,l)$
for $\varepsilon \leqslant\varepsilon_0(k,M_1,l)$.
We infer that for any $k\in\N$ there is a subsequence $\varphi_{k,\varepsilon_i}^{(3)}$ which weakly converges to $\tilde{f}_k^{(3)}\in H^1(M_1)$ and strongly in $L^2(M_1)$ and such that $\lim_i\lambda_{k}^{\varepsilon_i}=\alpha_k$. By definitions of $M_1^\varepsilon$ and $\varphi_{k,\varepsilon}^{(3)}$, and since $\mathcal{C}^\infty_0(M_1\setminus\{x_0\})$ is dense in $\mathcal{C}^\infty(M_1)$, it is easy to see that $\tilde{f}_k^{(3)}$ is a distributional (hence a strong) solution to $\Delta\tilde{f}_k^{(3)}=\alpha_k\tilde{f}_k^{(3)}$ on $M_1$ (see \cite{Tak1}, p.206). In particular, either $\tilde{f}_k^{(3)}$ is $0$ or $\alpha_k$ is an eigenvalue of $M_1$.

By the same compactness argument, there exists a subsequence $\varphi_{k,\varepsilon_i}^{(1)}$ which weakly converges to $\tilde{f}^{(1)}_k$ in $H^1(H_{1}\cup J_1)$ and strongly in $L^2(H_{1}\cup J_1)$. By Equalities \eqref{neg}, we get that $\|d\tilde{f}^{(1)}_k\|_{L^2(H_1)}=0$ and so $\tilde{f}^{(1)}_k$ is constant on $H_1$ and on $J_1$ and $\varphi_{k,\varepsilon_i}^{(1)}$ strongly converges to $\tilde{f}^{(1)}_k$ in $H^1(H_{1}\cup J_1)$. Let $\eta:[0,10]\to[0,1]$ be a smooth function such that $\eta(x)=1$ for any $x\leqslant1/2$, $\eta(x)=0$ for any $x\geqslant1$ and $|\eta'|\leqslant4$. We set $s_\varepsilon$ the distance function to $\partial S_\varepsilon=\{0\}\times\frac{\varepsilon}{10}\S^{m-1}$ in $S_\varepsilon=M_1^\varepsilon\cup J_\varepsilon$ and $\theta_\varepsilon$ the volume density of $S_\varepsilon$ in normal coordinate to $\partial S_\varepsilon$. We set $L$ the distance between the two boundary components of $J_1$. By construction of $S_\varepsilon$, we have $\frac{3}{10}\geqslant\theta_\varepsilon(s_\varepsilon,u)=\theta_1(s_\varepsilon/\varepsilon)\geqslant1$ for any $s_\varepsilon\in[0,L\varepsilon]$ and any $u$ normal to $\partial S_\varepsilon$, and $c(M_1)(\frac{s_\varepsilon}{\varepsilon})^{m-1}\geqslant\theta_\varepsilon(s_\varepsilon,u)\geqslant \frac{1}{c(M_1)}(\frac{s_\varepsilon}{\varepsilon})^{m-1}$ for $s_\varepsilon\in[\varepsilon L,8\varepsilon]$. Hence, if we denote by $S_{\partial S_\varepsilon}(r)$ the set of points in $S_\varepsilon$ at distance $r$ from $\partial S_\varepsilon$, we get for any $r\leqslant8+L$ that
\begin{align}\label{Spe}
\int_{S_{\partial S_\varepsilon}(\varepsilon r)}(f_k^\varepsilon)^2&=\int_{\frac{\varepsilon}{10}\S^{m-1}}\Bigl(\int_{\varepsilon r}^{1}\frac{\partial}{\partial s_\varepsilon}[\eta(\cdot)f_k^\varepsilon(\cdot,u)]d s_\varepsilon\Bigr)^2\theta_\varepsilon(r\varepsilon,u) du\nonumber\\
&=\frac{\varepsilon^{m-1}}{10^{m-1}}\int_{\S^{m-1}}\Bigl(\int_{\varepsilon r}^{1}\frac{\partial}{\partial s_\varepsilon}[\eta(\cdot)f_k^\varepsilon(\cdot,\frac{\varepsilon}{10} u)]d s_\varepsilon\Bigr)^2\theta_\varepsilon(r\varepsilon,\frac{\varepsilon}{10} u) du\nonumber\\
\leqslant \frac{c(M_1)\varepsilon^{m-1}}{10^{m-1}}&\int_{\S^{m-1}}\Bigl(\int_0^1\bigl(\frac{\partial}{\partial s_\varepsilon} [\eta(\cdot)f_k^\varepsilon(\cdot,\frac{\varepsilon}{10}u)]\bigr)^2 \theta_\varepsilon(s_\varepsilon,\frac{\varepsilon}{10}u)ds_\varepsilon\Bigr) \Bigl(\int_0^1\frac{1}{\theta_\varepsilon(s_\varepsilon,\frac{\varepsilon}{10}u)}ds_\varepsilon\Bigr) du\nonumber\\
&\ \ \ \ \ \ \ \ \ \ \int_{S_{\partial S_\varepsilon}(\varepsilon r)}(f_k^\varepsilon)^2\leqslant c(M_1)\|f_k^\varepsilon\|_{H^1(S_\varepsilon)}^2\varepsilon|\ln\varepsilon|
\end{align}
which gives us $\varepsilon_i\int_{\partial S_{\varepsilon_i}}(f_k^{\varepsilon_i})^2=\int_{\partial S_1}(\varphi_{k,\varepsilon_i}^{(1)})^2\to \int_{\partial S_1}(\tilde{f}_k^{(1)})^2=0$ (by the trace inequality and the compactness of the trace operator) and so $\tilde{f}_k^{(1)}$ is null on $J_1$.

By \eqref{T'}, and since $c$ is positive and $\mathcal{C}^1$ on $[0,l]$, there exists a subsequence $\varphi_{k,\varepsilon_i}^{(2)}$ which converges weakly to $\tilde{f}^{(2)}_k$ in $H^1([0,l]\times\S^{m-1})$ and strongly in $L^2([0,l]\times\S^{m-1})$. By the trace inequality applied on $[0,l]\times\S^{m-1}$, we also have that $\|\varphi^{(2)}_{k,\varepsilon_i}\|_{L^2(\{l\}\times\S^{m-1})}$ is bounded. Now, since 
$$10^{1-m}\varepsilon_i\int_{\{l\}\times\S^{m-1}}|\varphi^{(2)}_{k,\varepsilon_i}|^2=\varepsilon_i\int_{\{l\}\times \frac{\varepsilon_i}{10}\S^{m-1}}|f_k^{\varepsilon_i}|^2=\varepsilon_i\int_{\partial H_{\varepsilon_i}}|f_k^{\varepsilon_i}|^2=\int_{\partial H_{1}}|\varphi^{(1)}_{k,\varepsilon_i}|^2$$
we get that $\tilde{f}_k^{(1)}=0$ on $H_1$.

We set $h_i(t)=\int_{\S^{m-1}}\varphi_{k,\varepsilon_i}^{(2)}(t,x)dx$ and $h(t)=\int_{\S^{m-1}}\tilde{f}^{(2)}_k(t,x)dx$, we have $h,h_i\in H^1([0,l])$ (with $h_i'(t)=\int_{\S^{m-1}}\frac{\partial\varphi_{k,\varepsilon_i}^{(2)}}{\partial t}(t,x)dx$), $h_i\to h$ strongly in $L^2([0,l])$ and weakly in $H^1([0,l])$. For any $\psi\in\mathcal{C}^\infty([0,l])$ with $\psi(0)=0$ and $\psi'(l)=0$, seen as a function on $T_{c,\varepsilon}$ and extended by $0$ to $S_\varepsilon$ and by $\psi(l)$ to $H_\varepsilon$, we have
\begin{align*}&\int_0^lh'(\psi c^{m-1})'\,dt-(m-1)\int_0^lh'\frac{c'}{c}\psi c^{m-1}\,dt=\int_0^lh'\psi'c^{m-1}\,dt\\
&=\lim_i\int_0^lh_i'(t)\psi'(t)\frac{c^{m-1}}{\sqrt{1+\varepsilon_i^2(c')^2}}\,dt=\lim_i\int_{M_{\varepsilon_i}}\varepsilon_i^\frac{1-m}{2}\langle df_k^{\varepsilon_i},d\psi\rangle=\lim_i\int_{M_{\varepsilon_i}}\varepsilon_i^\frac{1-m}{2}\lambda_k^{\varepsilon_i}f_k^{\varepsilon_i}\psi\\
&=\alpha_k \lim_i\Bigl(\int_{[0,l]\times\S^{m-1}}\varphi^{(2)}_{k,\varepsilon_i}\psi c^{m-1}\sqrt{1+\varepsilon_i^2(c')^2}+\psi(l)\varepsilon_i^\frac{1-m}{2}\int_{H_{\varepsilon_i}}f_k^{\varepsilon_i}\Bigr)\\
&=\alpha_k\int_0^lh\psi c^{m-1}\,dt
\end{align*}
where we have used that $\varepsilon_i^\frac{1-m}{2}|\int_{H_{\varepsilon_i}}f_k^{\varepsilon_i}|\leqslant\sqrt{\varepsilon_i}\sqrt{\Vol(H_1)\int_{H_{\varepsilon_i}}(f_k^{\varepsilon_i})^2}$.
Since $c$ is positive, we get that $h$ is a weak solution to $y''+(m-1)\frac{c'}{c}y'+\alpha_ky=0$ on $[0,l]$ and that $h'(l)=0$.
 Since we have $10^{m-1}\int_{\partial S_{\varepsilon_i}}(f_k^{\varepsilon_i})^2=\int_{\{0\}\times\S^{m-1}}(\varphi_{k,\varepsilon_i}^{(2)})^2\to \int_{\{0\}\times\S^{m-1}}(\tilde{f}_k^{(2)})^2$ (by compactness of the trace operator) and $\int_{\partial S_{\varepsilon_i}}(f_k^{\varepsilon_i})^2\to0$ by \eqref{Spe}, we get $|h(0)|^2\leqslant\Vol\S^{m-1}\int_{\{0\}\times\S^{m-1}}(\tilde{f}_k^{(2)})^2=0$, and so $h(0)=0$. 
Since $d_{\S^{m-1}}\varphi_{k,\varepsilon_i}^{(2)}$ converges weakly to $d_{\S^{m-1}}\tilde{f}^{(2)}_k$ in $L^2([0,l]\times\S^{m-1})$, Inequality \eqref{T'} gives $\|d_{\S^{m-1}}\tilde{f}^{(2)}_k\|_{L^2([0,l]\times\S^{m-1})}=0$, i.e. $\tilde{f}^{(2)}_k$ is constant on almost every sphere $\{t\}\times\S^{m-1}$ of $[0,l]\times\S^{m-1}$. We infer that $\tilde{f}^{(2)}_k$ is equal to $\frac{1}{\Vol\S^{m-1}}h$ seen as a function on $[0,l]\times\S^{m-1}$ and so, either $\tilde{f}^{(2)}_k=0$ or $\alpha_k$ is an eigenvalue of $P_c$ for the Dirichlet condition at $0$ and the Neumann condition at $l$.

To conclude, we have
\begin{align*}
&\int_{M_1}\tilde{f}_k^{(3)}\tilde{f}_l^{(3)}+\int_{[0,l]\times\S^{m-1}} \tilde{f}_k^{(2)}\tilde{f}_l^{(2)}c^{m-1}\\
&=\lim_i\int_{M_1}\varphi_{k,\varepsilon_i}^{(3)}\varphi_{l,\varepsilon_i}^{(3)}+\int_{J_1\cup H_1}\varphi_{k,\varepsilon_i}^{(1)}\varphi_{l,\varepsilon_i}^{(1)}+\int_{[0,l]\times\S^{m-1}} \varphi_{k,\varepsilon_i}^{(2)}\varphi_{l,\varepsilon_i}^{(2)}c^{m-1}\sqrt{1+\varepsilon_i^2(c')^2}\\
&=\lim_i\int_{M_{\varepsilon_i}}f_k^{\varepsilon_i}f_l^{\varepsilon_i}-\lim_i\int_{M_1^{\varepsilon_i}\cap B(\partial M_1^{\varepsilon_i},8\varepsilon_i)}f_k^{\varepsilon_i}f_l^{\varepsilon_i}+\lim_i\int_{M_1\setminus\bigl( M_1^{\varepsilon_i}\setminus B(\partial M_1^{\varepsilon_i},8\varepsilon_i)\bigr)}\varphi^{(3)}_{k,\varepsilon_i}\varphi^{(3)}_{l,\varepsilon_i}\\
&=\delta_{kl},
\end{align*}
where, in the last equality, we have used that $\varphi^{(3)}_{k,\varepsilon_i}$ and $\varphi^{(3)}_{l,\varepsilon_i}$ converge strongly to $\tilde{f}_k^{(3)}$ and $\tilde{f }_l^{(3)}$ in $L^2(M_1)$, that $\Vol \bigl(M_1\setminus(M_1^{\varepsilon_i}\setminus B(\partial M_1^{\varepsilon_i},8\varepsilon_i)\bigr)$ tends to $0$ with $\varepsilon_i$, and the inequality $$\int_{M_1^{\varepsilon_i}\cap B(\partial M_1^{\varepsilon_i},8\varepsilon_i)}(f_k^{\varepsilon_i})^2\leqslant c(M_1)\| f_k^{\varepsilon_i}\|_{H^1(M_{\varepsilon_i})}\varepsilon_i^2|\ln\varepsilon_i|$$
which is obtained by integration of Inequality \eqref{Spe} with respect on $r\in[L,L+8]$. Note that we need the inclusion Hence, by the min-max principle, we have $\alpha_k\geqslant\lambda_k$ for any $k\in\N$.
We conclude that $\lim_{\varepsilon\to 0}\lambda_k(M_\varepsilon)=\lambda_k$ for any $k\in\N$.
Note that in the case $c\equiv\frac{1}{10}$,
the spectrum of $P_c$ with Dirichlet condition at $0$ and Neumann condition at $l$ is $\{\frac{\pi^2}{l^2}(k+\frac{1}{2})^2,\,k\in\N\}$ with all the multiplicities equal to $1$. 

\noindent\underline{End of the proof of Theorem \ref{constrfond}}{~}\\
In the sequence of immersions constructed above we have all the properties announced for $T=[x_0,x_0+l\nu]$, except the point 3) since all the eigenvalues of $[0,l]$ appear in the spectrum of the limit. To get the result for $T=[x_0,x_0+l\nu]$, we fix $k\in\N$ and $l_k$ small enough such that $\lambda_1([0,l_k])>2k$ and with $l/l_k\in\N$. We then consider an immersion of $N_1=M_1\#\S^m$ such that $d_H(M_1\cup[x_0,x_0+l_k\nu],N_1)\leqslant2^{-\frac{l}{l_k}}$, $|\lambda_p(N_1)-\lambda_p(M_1)|\leqslant2^{-\frac{l}{l_k}}$ for any $p$ such that $\lambda_p(M_1)\leqslant k$ (it is possible for this choice of $l_k$ according to the weak version of the theorem proved above) and the same for the point 4) and 2) of the theorem (equality up to an error bounded by $2^{-\frac{l}{l_k}}$). We now iterate the procedure to get a sequence of $\frac{l}{l_k}$ immersions $N_2=N_1\#\S^m,\cdots,\ N_{\frac{l}{l_k}-1}=N_{\frac{l}{l_k}-2}\#\S^m$, $N_\frac{l}{l_k}=N_{\frac{l}{l_k}-1}\# M_2$ such that 
\begin{align*}
&d_H(N_{i},M_1\cup[x_0,x_0+il_k\nu])\leqslant i 2^{-\frac{l}{l_k}},\ \ \ |\Vol N_{i+1}\setminus N_i|\leqslant 2^{-\frac{l}{l_k}}\Vol M_1,\\
&\int_{N_{i+1}\setminus N_i}|\B|^{(m-1)\frac{k-1}{k}}\leqslant 2^{-\frac{l}{l_k}}\int_{M_1}|\H|^{(m-1)\frac{k-1}{k}},\\
&|\int_{N_{i+1}\setminus N_i}|\B|^{m-1}-\Vol\S^{m-1}l_k|\leqslant 2^{-\frac{l}{l_k}}\int_{M_1}|\B|^{m-1},\\
&|\int_{N_{i+1}\setminus N_i}|\H|^{m-1}-\Vol\S^{m-1}(\frac{m-1}{m})^{m-1} l_k|\leqslant 2^{-\frac{l}{l_k}}\int_{M_1}|\H|^{m-1},\\
&|\lambda_p(N_i)-\lambda_p(N_{i+1})|\leqslant 2^{-\frac{l}{l_k}}\mbox{ for any }i\leqslant \frac{l}{l_k}-1\mbox{ and any }p\leqslant N.
\end{align*} 
The sequence $i_k(M_1\# M_2)=N_\frac{l}{l_k}$ satisfies Theorem \ref{constrfond} for $T=[x_0+l\nu]$.

In the procedure to get the theorem for $T=[x_0,x_0+l\nu]$ from its weak version, we add at each step the new small cylinder $T_{\varepsilon,l_N}$ along the same axis $x_0+\R_+\nu$. This can be easily generalized to get the lemma for $T=\cup_i T_i$ any finite union of finite trees, each intersecting $M_1$, and such that $\sum_im_1(T_i)\leqslant l$. Finally, if $T$ is a closed subset such that $m_1(T)\leqslant C(m)\int|\H|^{m-1}$ and $M_1\cup T$ is connected, then each connected component of $T$ intersects $M_1$. Arguing as in the proof of Theorem \ref{Main}, we get that $T$ has only a finite number of connected component leaving $(M_1)_\frac{1}{k}$. Since any closed, connected $F_i\subset\R^{n+1}$ with $m_1(F_i)$ finite can be approximated in Hausdorff distance by a sequence of finite trees $T_{i,k}$ (such that $m_1(T_{i,k})\to_k m_1(F_i)$ see \cite{Fa}), by the same kind of diagonal procedure, Theorem \ref{Main} is obtained for any $T$ with finite $m_1(T)$. In the case $m_1(T)=\infty$ then the $L^{m-1}$ control of the curvature in condition 2) are automatically satisfied and the other conditions are fulfilled as above by approximating $M_1\cup T$ by a finite number of finite trees. Finally, if $M_1\cup T$ is not bounded, then we replace it by $\bigl((M_1\cup T)\cap B_0(k)\bigr)\cup k\S^n$, which is closed, connected and bounded. We then get the result in its whole generality by a diagonal procedure.


\begin{thebibliography}{99}
\bibitem{Au} {\sc E. Aubry}, {\em Pincement sur le spectre et le volume en courbure de Ricci positive},  Ann. Sci. \'{E}cole Norm. Sup. (4) {\bf 38} (2005), n°3, p. 387--405.
\bibitem{AG1} {\sc E. Aubry, J.-F. Grosjean}, {\em Spectrum of hypersurfaces with small extrinsic radius or large $\lambda_1$ in Euclidean spaces}, preprint (2012) arXiv:?.
\bibitem{AGR1} {\sc E. Aubry, J.-F. Grosjean, J. Roth}, {\em Hypersurfaces with small extrinsic radius or large $\lambda_1$ in Euclidean spaces}, preprint (2010) arXiv:1009.2010v1.
\bibitem{Ann} {\sc C. Ann\'{e}}, {\em Spectre du Laplacien et \'{e}crasement d'anses}, Ann. Sci. \'{E}cole Norm. Sup. (4) {\bf 20} (1987), p. 271--280.
\bibitem{colgros} {\sc B. Colbois, J.-F. Grosjean}, {\em A pinching theorem for the first eigenvalue of the Laplacian on hypersurfaces of the Euclidean space}, Comment. Math. Helv. {\bf 82}, (2007), p. 175--195.
\bibitem{HK} {\sc T. Hasanis, D. Koutroufiotis}, {\em Immersions of bounded mean curvature}, Arc. Math. {\bf 33}, (1979), p. 170--171.
\bibitem{Fa} {\sc K.J. Falconer} {\em The geometry of fractal sets} Cambridge 1985.
\bibitem{GR} {\sc J.-F. Grosjean, J. Roth} {\em Eigenvalue pinching and application to the stability and the almost umbilicity of hypersurfaces}, to appear in Math. Z.
\bibitem{Mi-Si} {\sc J.H. Michael, L.M. Simon}
{\em Sobolev and mean-value inequalities on generalized submanifolds of $R\sp{n}$}, Comm. Pure Appl. Math. {\bf 26} (1973), p. 361--379.
\bibitem{RT} {\sc J. Rauch, M. Taylor} {\em Potential and scattering theory on wildly perturbed domains}, J. Func. Anal. 18 (1975), p. 27--59.
\bibitem{Re} {\sc R.C. Reilly}, {\em On the first eigenvalue of the Laplacian for compact submanifolds of Euclidean space}, Comment. Math. Helv. {\bf 52}, (1977), p. 525--533.
\bibitem{roth} {\sc J. Roth}, {\em Extrinsic radius pinching for hypersurfaces of space forms}, Diff. Geom. Appl. {\bf 25}, No 5, (2007), P. 485--499.
\bibitem{Tak1} {\sc J. Takahashi}, {\em Collapsing of connected sums and the eigenvalues of the Laplacian}, J. Geom. Phys. {\bf 40} (2002), p. 201--208.
\bibitem{Top} {\sc P. Topping}, {\em Relating diameter and mean curvature for submanifolds of Euclidean space}, Comment. Math. Helv. {\bf 83}, (2008), no. 3, p. 539--546.
\end{thebibliography}
\end{document}